\documentclass[11pt]{article}
\usepackage{tikz}
\usepackage{subfigure}
\usepackage[english]{babel}
\usetikzlibrary{arrows, graphs}
\usepackage[center]{caption2}
\usepackage{amsfonts,amssymb,amsmath,latexsym,amsthm}
\usepackage{multirow}
\usepackage{bbm}
\usepackage[colorlinks=true,allcolors=blue]{hyperref}
\usepackage[noabbrev,capitalize,nameinlink]{cleveref}

\newtheorem{thm}{Theorem}[section]
\newtheorem{cor}[thm]{Corollary}
\newtheorem{lem}[thm]{Lemma}
\newtheorem{prop}[thm]{Proposition}

\begin{document}

\title{Some Ramsey-type results
}
\author{Jin Sun\footnote{\texttt{e-mail:jinsun@mail.ustc.edu.cn}}\\
\small  School of Mathematical Sciences\\
\small University of Science and Technology of China, Hefei, Anhui 230026, China.\\
}
\date{}

\maketitle 	
\begin{abstract}The Ramsey's theorem says that a graph with sufficiently many vertices contains a {clique} or {independent set} with many vertices. Now we attach some parameter to every vertex, such as {degree}. Consider the case a graph with sufficiently many vertices of large degree, we can get a similar Ramsey-type result. The Ramsey's theorem of connected version says that every connected graph with sufficiently many vertices contains an { induced path}, { clique} or { star} with many vertices. Now we require the vertex is non-trivial, i.e. the parameter of this vertex is non-trivial, such as $\operatorname{deg}(v)\ge 2$. A connected graph with sufficiently many non-trivial vertices must contain some special induced subgraph. We also get the non-connected version of this Ramsey-type result as a corollary. 
\end{abstract}

\section{Introduction}
For positive integers $m$ and $n$ with $m<n$, let $[m,n]$ denote the set $\{m,m+1,\dots ,n-1,n\}$, and let $[n]$ denote $[1,n]$.
In this paper, we consider only finite, undirected and simple graphs. Let $G$ be a graph. For vertices $y,z\in V(G)$, we use the symbol $y\sim z$ to mean that  $y$ and $z$ are adjacent. $\mathbbm 1_{y\sim z}$ is the {\it characteristic function} which equals $1$ if $y$ is adjacent to $z$, $0$ otherwise.
For a vertex $v\in V(G)$, let $N(v)$ denote the {\it neighborhood} of $v$ in $G$, thus $N(v)=\{u\in V(G): u\sim v\}$. For a subset $S\subset V(G)$, we denote by $G[S]$ the subgraph of $G$ induced by $S$, i.e. $V(G[S])=S,\, E(G[S])=\{\{u,v\}: u,v\in S, u\sim v \mbox{ in} \;G\}$. Let $\alpha(S)$ be the {\it independence number} of the induced subgraph $G[S]$, $c(S)$ be the {\it number of connected components} of $G[S]$. We say $v$ is {\it complete} to $S$ if $S\subset N(v)$. For two disjoint subsets $A, B$ of $V(G)$, we denote by $G[A,B]$ the bipartite subgraph (not induced if $E(G[A])\cup E(G[B])\neq \emptyset$.) with two parts $A$ and $B$ formed by the edges of $G$ between $A$ and $B$, i.e. $V(G[A,B])=A\sqcup B$, $E(G[A,B])=\{ \{a,b\}: a\in A,\, b\in B,\, a\sim b \mbox{ in } G \}$. {\it Stable set} has the same meaning as {\it independent set}. The more notations, see \cite{d}.

For graphs $H_1, H_2$, we say $H_1\le H_2$ if $H_1$ is an induced subgraph of $H_2$. For a family $\mathcal{H}$ of graphs, we say $G$ is {\it $\mathcal{H}$-free}, if there is no graph $H\in \mathcal{H}$ such that $H\le G$.
For families $\mathcal{H}_{1},\, \mathcal{H}_2$, we say $\mathcal{H}_1\le \mathcal{H}_2 $ if any graph in $\mathcal{H}_2$ is not $\mathcal{H}_1$-free, i.e. $\forall \,H_2\in \mathcal{H}_2, \, \exists \, H_1\in \mathcal{H}_1$ such that $ H_1\le H_2$. Specially, $H_1\le H_2$ is equivalent to $\{H_1\}\le \{H_2\}$.
Note that if $\mathcal{H}_1\le \mathcal{H}_2$, then every $\mathcal{H}_1$-free graph is also $\mathcal{H}_2$-free.
This language was introduced in \cite{f06} firstly.

In 1929, the well-known Ramsey's theorem appeared.

\begin{thm}\cite r 
For any positive integers $m,n$, there exists a minimum positive integer $R=R_{m}(n)$ such that if the edges of the complete graph $K_R$ are colored with $m$ colors, then there is a monochromatic clique of order $n$, i.e. a clique all of whose edges have the same color.
\end{thm}

Specially, for $m=2$, Ramsey's theorem says that for any large $n$, a graph with $R_{2}(n)$ vertices contains  $K_n$ or $E_n$ as an induced subgraph. $E_n$ is not connected. When we restrict our attention to connected graphs, we can get connected induced subgraph better than $E_n$.  

\begin{prop}\cite d \label{d}
For any positive integer $n$, there exists an integer $N_{0}(n)$ such that every connected graph on at least 
$N_0(n)$ vertices contains $P_n, K_n$ or $K_{1,n}$ as an induced subgraph.
\end{prop}

Lozin \cite{l17} introduces a general philosophy of Ramsey-type problems: given a graph parameter $\mu$,  characterize the forbidden graphs condition to bound the parameter $\mu$, i.e. find all graph families $\mathcal H$ such that there is a constant $c$ satisfying $\mu(G)\le c$ for any connected $\mathcal{H}$-free graph $G$. When $\mu$ is the order of graph $G$, this philosophy leads to the classical Ramsey's theorem.  Atminas, Lozin, and Razgon \cite{a} studied the parameter $\mu$ as the length of longest path in graph $G$. Chiba and Furuya \cite{c22} studied the parameter path cover number. Furuya \cite{f18} studied the parameter domination number. Many other parameters have been studied in \cite{al,c20,l17}.

Now we consider the parameter $\mu$ as the variation of the order of graph $G$. Compared to Proposition \ref{d}, we restrict that the graph has many nontrivial vertices to get better induced subgraphs. Here we attach some parameter $p$ to every vertex and  vertex $v$ is nontrivial if  $p(v)\ge 2$. Specifically,  $p(v)$ will be $\deg(v)$, $\alpha(N(v))$, $c(N(v))$, and the {\it adhesion} of $v$. We define the adhesion of $v$ as $c(G-v)-c(G)+1$, and denote it by $\operatorname{adh}(v)$. Thus $v$ is a {\it cut vertex} if and only if $\operatorname{adh}(v)\ge 2$. Therefore the adhesion of a vertex is a generalization of cut vertex. Notice that for a vertex $v\in V(G)$ , we have $\operatorname{deg}(v)\ge \alpha(N(v))\ge c(N(v))\ge \operatorname{adh}(v)$. 

In order to state our Ramsey-type results, we introduce some kinds of graphs. As usual, we let $K_n,E_n,P_n$ and $K_{s,t}$ denote the {\it complete graph}, {\it edgeless graph}, {\it path} and {\it complete bipartite graph} respectively. For two graphs $G_1$ and $G_2$, we define the {\it join} $G_1+G_2$ by the graph with vertex set $V(G_1)\cup V(G_2) $ and edge set $E(G_1)\cup E(G_2)\cup \{ \{x,y\} :x\in V(G_1),\, y\in V(G_2) \} $. We denote by $nG$ the graph consisting of $n$ disjoint copies of $G$. The more kinds of graphs, see \cref{fig 1}. 
\begin{itemize}
	\item $K_{1,n}^{*} $: The graph obtained by adding a pendant to every leaf of $K_{1,n}$.
	\item $K_n^{*}$: The graph obtained by adding a pendant to every vertex of $K_n$.
	\item $CK_n$: The graph obtained by adding a perfect matching between two disjoint copies of $K_n$.
	\item $T_n$: The graph obtained by adding a new vertex adjacent to all vertices of $E_n$ from $K_n+E_n$.
	\item $K_n^n$: The graph obtained by adding $n$ pendents to every vertex of $K_n$.
\end{itemize}

\begin{figure}
\centering
\begin{tikzpicture}[baseline=10pt,scale=0.8]
\filldraw[fill=black,draw=black] (0,0) circle (0.1);
\filldraw[fill=black,draw=black] (1,2) circle (0.1);
\filldraw[fill=black,draw=black] (1,1) circle (0.1);
\node at (1,-0.2) {$\vdots$};
\filldraw[fill=black,draw=black] (1,-1.5) circle (0.1);
\filldraw[fill=black,draw=black] (2,2) circle (0.1);
\filldraw[fill=black,draw=black] (2,1) circle (0.1);
\filldraw[fill=black,draw=black] (2,-1.5) circle (0.1);
\draw (0,0) -- (1,2);
\draw (0,0) -- (1,1);
\draw (0,0) -- (1,-1.5);
\draw (2,2) -- (1,2);
\draw (2,1) -- (1,1);
\draw (2,-1.5) -- (1,-1.5);
\node at (1,-3) {$K_{1,n}^*$};

\draw (3.5,-1.8) rectangle (4.5,2.3);
\filldraw[fill=black,draw=black] (4,-1.5) circle (0.1);
\filldraw[fill=black,draw=black] (4,1) circle (0.1);
\filldraw[fill=black,draw=black] (4,2) circle (0.1);
\node at (5,-0.2) {$\vdots$};
\node at (4,-0.3) {$K_n$};
\filldraw[fill=black,draw=black] (5,-1.5) circle (0.1);
\filldraw[fill=black,draw=black] (5,1) circle (0.1);
\filldraw[fill=black,draw=black] (5,2) circle (0.1);
\draw (4,-1.5) -- (5,-1.5);
\draw (4,1) -- (5,1);
\draw (4,2) -- (5,2);
\node at (4,-3) {$K_{n}^*$};

\draw (7.5,-1.8) rectangle (6.5,2.3);
\draw (8.5,-1.8) rectangle (9.5,2.3);
\filldraw[fill=black,draw=black] (7,-1.5) circle (0.1);
\filldraw[fill=black,draw=black] (9,-1.5) circle (0.1);
\filldraw[fill=black,draw=black] (7,1) circle (0.1);
\filldraw[fill=black,draw=black] (9,1) circle (0.1);
\filldraw[fill=black,draw=black] (7,2) circle (0.1);
\filldraw[fill=black,draw=black] (9,2) circle (0.1);
\draw (7,-1.5) -- (9,-1.5);
\draw (7,1) -- (9,1);
\draw (7,2) -- (9,2);
\node at (9,-0.2) {$\vdots$};
\node at (7,-0.3) {$K_n$};
\node at (8,-3) {$CK_n$};

\draw (12,-1.8) rectangle (11,2.3);
\filldraw[fill=black,draw=black] (11.5,-1.5) circle (0.1);
\filldraw[fill=black,draw=black] (11.5,1) circle (0.1);
\filldraw[fill=black,draw=black] (11.5,2) circle (0.1);
\node at (12.5,-0.2) {$\vdots$};
\node at (11.5,-0.3) {$K_n$};
\filldraw[fill=black,draw=black] (12.5,-1.5) circle (0.1);
\filldraw[fill=black,draw=black] (12.5,1) circle (0.1);
\filldraw[fill=black,draw=black] (12.5,2) circle (0.1);
\draw (11.5,-1.5) -- (12.5,-1.5);
\draw (11.5,1) -- (12.5,1);
\draw (11.5,2) -- (12.5,2);
\draw (11.5,-1.5) -- (12.5,1);
\draw (11.5,1) -- (12.5,-1.5);
\draw (11.5,2) -- (12.5,-1.5);
\draw (11.5,-1.5) -- (12.5,2);
\draw (11.5,1) -- (12.5,2);
\draw (11.5,2) -- (12.5,1);
\filldraw[fill=black,draw=black] (13.5,0) circle (0.1);
\draw (13.5,0) -- (12.5,2);
\draw (13.5,0) -- (12.5,1);
\draw (13.5,0) -- (12.5,-1.5);
\node at (12,-3) {$T_n$};

\draw (17,0) circle (1.5);
\node at (17,0) {$K_n$};
\filldraw(17,1) circle (0.1);
\filldraw(16.5,2) circle (0.1);
\filldraw(17,2) circle (0.1);
\node at (17.5,2) {$\cdots$};
\filldraw(18,2) circle (0.1);
\draw (17,1) --  (16.5,2);
\draw (17,1) --  (17,2);
\draw (17,1)--   (18,2);
\node at (15,0) {$\vdots$};
\node at (19,0) {$\vdots$};
\filldraw(17,-1) circle (0.1);
\filldraw(16.5,-2) circle (0.1);
\filldraw(17,-2) circle (0.1);
\node at (17.5,-2) {$\cdots$};
\filldraw(18,-2) circle (0.1);
\draw (17,-1) --  (16.5,-2);
\draw (17,-1) --  (17,-2);
\draw (17,-1)--   (18,-2);
\node at (17,-3) {$K_n^n$};
\end{tikzpicture}
\caption{The graphs $K_{1,n}^*,\, K_n^*,\, CK_n,\, T_n$ and $K_n^n.$}
\label{fig 1}
\end{figure}
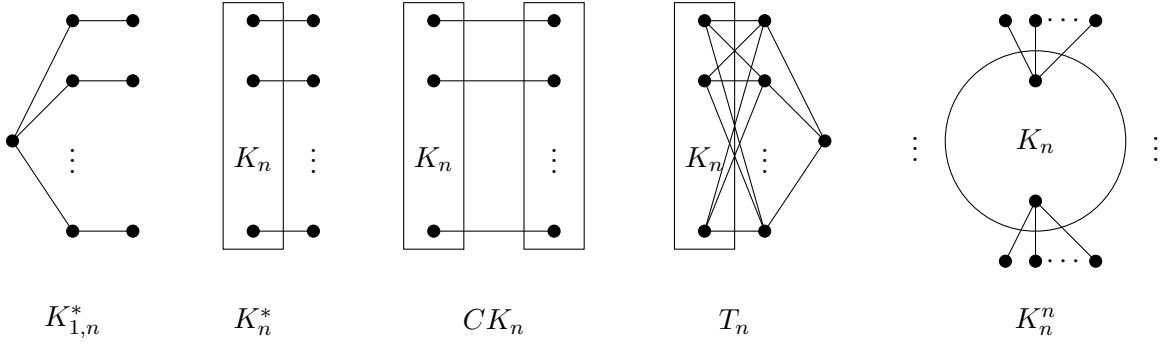

Requiring the graph has many nontrivial vertices, we have the following related Ramsey-type results.

\begin{thm}\label{deg} There is a constant $c=c(\mathcal H)$ such that
 $$\#\{v\in V(G):\deg(v)\geq 2\}\le c $$ 
 for every connected $\mathcal {H}$-free graph $G$ if and only if 
 \[\mathcal{H}\le \{K_n, P_n, K_{1,n}^*, K_{2,n}, K_2+nK_1, K_1+nK_2 \} \]
 for some positive integer $n$.
\end{thm}

\begin{thm}\label{local stable number}
There is a constant $c=c(\mathcal H)$ such that 
$$\#\{v\in V(G):\alpha(N(v)) \geq 2\} \le c $$
 for every connected $ \mathcal {H}$-free graph $G$ if and only if 
 $$ \mathcal{H}\le \{K_n^*, P_n, K_{1,n}^*, K_{2,n}, E_2+K_n, K_1+nP_3, CK_n \} $$ for some positive integer $n$.
\end{thm}

\begin{thm}\label{local connected} 
There is a constant $c=c(\mathcal H)$ such that 
$$\#\{v\in V(G):c(N(v)) \geq 2\} \le c $$
 for every connected $\mathcal H$-free graph $G$ if and only if 
 $$ \mathcal{H}\le \{K_n^*, P_n, K_{1,n}^*, K_{2,n}, CK_n, T_n \} $$
 for some positive integer $n$.
\end{thm}

\begin{thm}\label{cut vertex}
There is a constant $c=c(\mathcal H)$ such that 
\[\#\{v\in V(G): \operatorname{adh}(v)\geq 2 \}\le c\]
for every connected $\mathcal H$-free graph $G$
if and only if \[\mathcal{H}\le \{K_n^*, K_{1,n}^*, P_n\} \] 
for some positive integer $n$.
\end{thm}

 To be more precisely, we give another example. Using Ramsey's theorem easily, we have the following.

 \begin{prop} There is a constant $c=c(\mathcal H)$ such that $\Delta(G)\le c $ for every  $\mathcal {H}$-free graph $G$ if and only if $\mathcal{H}\le \{K_n,K_{1,n}\} $ for some positive integer $n$.
\end{prop}

We know that for any positive integer $n$, a graph with sufficiently large maximum degree has sufficiently many vertices and  hence contains $K_n$ or $E_n$ as an induced subgraph. However, $\Delta(E_n)=0$ is not large. Thus for family $\mathcal{H}$ of graphs, if $\mathcal{H}$ satisfies the condition that there is a constant $c=c(\mathcal H)$ such that $\Delta(G)\le c $ for every  $\mathcal {H}$-free graph $G$, we can not get the conclusion that $\mathcal{H}\le \{K_n,E_n\}$. The right conclusion is $\mathcal{H}\le \{K_n,K_{1,n}\} $. Here $K_{1,n}$ is a better induced subgraph than $E_n$. This example gives an explanation what the better means.

Let us continue to look at this example. Observe that for the star $K_{1,n}$, it has only one vertex with large degree, and other vertices are degree $1$. In some sense, it is unfair for most vertices to say $K_{1,n}$ has large maximum degree. Now we consider the case where there are many vertices of large degree.

\begin{thm}\label{finite, deg}
There are two constants $c_1=c_1(\mathcal{H} ), c_2=c_2(\mathcal{H})$ such that 
\[ \#\{v\in V(G): \deg(v)\ge c_1\} <c_2 \]
for every $\mathcal{H}$-free graph $G$, if and only if $\mathcal{H}\le \{K_n, K_{n,n}, nK_{1,n}\}$ for some positive integer $n$.
\end{thm}
Note that \cite{al} got the same result as Theorem \ref{finite, deg} in the language of minimal hereditary classes.  For a parameter $p$ of $V(G)$, the {\it $h$-index} of $G$ is the largest positive integer $k$ such that $G$ has $k$ vertices of parameter $p$ at least $k$. It's introduced in \cite{h} and important in the study of dynamic algorithms \cite{e}. When the parameter $p$ is vertex degree, Theorem \ref{finite, deg} characterizes the family $\mathcal H$ such that there is a uniform bound of $h$-index for any $\mathcal H$-free graph $G$. Similarly, we replace degree by other parameters of a vertex, then get the following.
\begin{thm}\label{finite, local stable number}
The following are equivalent:

(1) There are two constants $c_1=c_1(\mathcal{H} ), c_2=c_2(\mathcal{H})$ such that 
$$\#\{v\in V(G): \alpha(N(v))\ge c_1\} <c_2 $$
for every $\mathcal{H}$-free graph $G$.

(2) There are two constants $c_1=c_1(\mathcal{H} ), c_2=c_2(\mathcal{H})$ such that 
\[ \#\{v\in V(G): c(N(v))\ge c_1\} <c_2 \]
for every $\mathcal{H}$-free graph $G$.

(3) $\mathcal{H}\le \{K_{n,n}, nK_{1,n}, K_n+E_n, K^n_n\}$ for some positive integer $n$.
\end{thm}

\begin{thm}\label{finite adh}
There are two constants $c_1=c_1(\mathcal{H} ), c_2=c_2(\mathcal{H})$ such that 
\[ \#\{v\in V(G): adh(v)\ge c_1\} <c_2 \]
for every $\mathcal{H}$-free graph $G$, if and only if $\mathcal{H}\le \{nK_{1,n}, K^n_n\}$ for some positive integer $n$.
\end{thm}

In section $2$, we will consider the case a connected graph $G$ has many nontrivial vertices, and 
prove Theorem \autoref{deg}, Theorem \ref{local stable number}, Theorem \ref{local connected} and Theorem \ref{cut vertex}. 
In section $3$, we will consider the case a graph $G$ has many vertices of large parameter, and prove Theorem \ref{finite, deg}, Theorem \ref{finite, local stable number} and Theorem \ref{finite adh}. 
In section $4$, we will consider the case a graph has many nontrivial vertices, and give the corresponding corollaries. We will also give some remarks of our results.

\section{Connected graph with many nontrivial vertices}
The following lemma was implicitly proved and used in \cite{c22} and \cite{f18}.
Actually, it can date back to earlier, the definition of {\it irredundant set}.  Here we put it in the spotlight.

\begin{lem}\label{induced matching}
For a bipartite graph $G=(X,Y,E)$, the minimum degree of vertices in $X$ satisfies $\delta(X)\ge 1$, the maximum degree of vertices in $Y$ satisfies $\Delta(Y)\le n$. The induced matching number of $G$ satisfies $\mu' (G)\ge p$ 
provided that $|X|\ge n(p-1)+1 $, and the bound is tight.
\end{lem}

\begin{proof}
Since $\delta(X)\ge 1$, $Y$ dominates $X$, i.e. $X\subset N(Y)$. Take minimal subset $Y'$ of $Y$ dominates $X$, i.e. $N(Y')=X$ and for any proper subset $Y''$ of $Y'$, $N(Y'')\not =X$. So $Y'$ is irredundant, i.e. every $y\in Y'$ has a private neighbor $x_y\in X$ such that $N(x_y)\cap Y'=\{y\}$. Since $\Delta(Y)\le n$, every vertex can dominate at most $n$ vertices in $X$. $|X|\ge n(p-1)+1$, thus $|Y'|\ge p$. Choose $p$ vertices in $Y'$, and a private neighbor for everyone, they form an induced $pK_2$. So $\mu' (G)\ge p$.  

Specially, let $G=(X,Y,E)$ be $X=\{x_{ij}:i\in [p-1],\, j\in [n]\}$, $Y=\{y_i:i \in [p-1]\}$, $E=\{(x_{ij},y_i):i \in [p-1],\, j\in [n]\}$. It satisfies $\delta(X)\ge 1$, $\Delta(Y)\le n$, but $\mu' (G)= p-1$ since $|X|$ only equals to $ n(p-1) $.
\end{proof}

\begin{proof}[Proof of  Theorem \ref{deg}]{}
 
We first prove the ``only if'' part. 
Suppose $\mathcal H$ satisfies there is a constant $c=c(\mathcal H)$ such that $\#\{v\in V(G):\deg(v)\ge 2\}\le c $ for every connected $\mathcal {H}$-free graph $G$. The number of vertices with degree at least $2$ for 
$$K_n, P_n, K_{1,n}^*, K_{2,n}, K_2+nK_1, K_1+nK_2$$
is $n, n-2, n+1, n+2, n+2, 2n+1 $ respectively. Let $n= c+3$, then all these graphs are connected but the number of vertices with degree at least $2$ is not bounded by $c$, so they are not $\mathcal{H}$-free. Hence $\mathcal{H}\le \{K_n, P_n, K_{1,n}^*, K_{2,n}, K_2+nK_1, K_1+nK_2 \} $.

Next we prove the ``if'' part. Now $\mathcal{H}\le \{K_n, P_n, K_{1,n}^*, K_{2,n}, K_2+nK_1, K_1+nK_2 \} $. Since $\mathcal H$-free graph is also $\{K_n, P_n, K_{1,n}^*, K_{2,n}, K_2+nK_1, K_1+nK_2 \}$-free,  we need to show that if a connected graph $G$ satisfies 
$$\#\{v\in V(G):\deg(v)\ge 2\}> N_0(N_1), \, N_1=(n-1)(R_2(2n-1)) ,$$   
then it contains $K_n, P_n, K_{1,n}^*, K_{2,n}, K_2+nK_1$ or  $K_1+nK_2$ as an induced subgraph. ($N_0$ comes from  Proposition \ref{d}.) For a cut vertex $v$ of $H\le G$, an induced subgraph of $G$, it is connected to at least $2$ components of $H-v$, so $v$ has at least $2$ neighbors in $H$, $\deg_G(v)\ge 2$. Deleting all the vertices of degree $1$ in $G$ one by one, 
the left induced subgraph $G'$ with order at least $N_0(N_1)$ is still connected. So by Proposition \ref{d}, it contains $K_{N_1}, P_{N_{1}} $ or $K_{1,N_1} $. As $N_1\ge n$, we are done if $G'$ contains $K_{N_1}$, or $P_{N_{1}}$.

We may assume $G'$ contains $K=K_{1,N_1}$. Denote the central vertex by $s$, the leaves by $X=\{x_i: i\in [N_1]\}$. Since $\deg_G(x_i)\ge 2$ in $G$, we can find another vertex $y_i\not =s$ in $G$ adjacent to $x_i$. (Note that $y_i$ and $y_j$, $i,j\in [N_1],\, i\not =j$ maybe the same vertex.) If there is a $y_i$, satisfying $|N(y_i)\cap X|\ge n$, then $\{s,y_i, N(y_i)\cap X\}$ contains $K_{2,n}$ or $K_2+nK_1$ as an induced subgraph according to whether $y_i$ is adjacent to $s$.

Hence we may assume $|N(y_i)\cap X|<n, \; \forall i\in [N_1].$ Denote by $Y$ the set consisting of vertices $y_i, i\in[N_1]$. Consider the bipartite graph $G'[X,Y]$. By Lemma \ref{induced matching}, we can find $X'\subset X,\, Y'\subset Y$ which form an induced matching of order $R_2(2n-1)$ in $G'[X,Y]$. If $Y'$ contains a clique of order $n$, then we are done. Since $|Y'|=R_2(2n-1)$, $Y'$ contains a stable set $S$ of order $2n-1$. Using pigeonhole principle, if $|N(s)\cap S|\ge n$, then $G$ contains an $K_1+nK_2$ as an induced subgraph. Otherwise, $|N(s)\cap S|< n$, there are at least $n$ vertices not adjacent to $s$, so $G$ contains $K_{1,n}^*$.

In other words, every connected $\{K_n, P_n, K_{1,n}^*, K_{2,n}, K_2+nK_1, K_1+nK_2 \} $-free graph $G$ satisfies 
$\#\{v\in V(G):\deg(v)\ge 2\}\le c=(n-1)R_2(2n-1)$.  The proof is completed. 
 \end{proof}

\begin{proof}[Proof of Theorem \ref{local stable number}]
For a graph $G$, the vertex $v$ such that $\alpha(N(v))=1$ is not a cut vertex of any induced subgraph of $G$ which contains $v$. Deleting all these vertices one by one, the left induced subgraph $G'$ is still connected. 
 
We first prove the ``only if'' part. For $G=K_n^*, P_n, K_{1,n}^*, K_{2,n}, E_2+K_n, K_1+nP_3$, or  $CK_n$, the order of $G'$ is  $n, n-2, n+1, n+2, n, n+1, 2n$ respectively. Let $n= c+3$, we know that $ \mathcal{H}\le \{K_n^*, P_n, K_{1,n}^*, K_{2,n}, E_2+K_n, K_1+nP_3, CK_n \} $.

Next we prove the ``if'' part. Set $$N_1=2N_2-1,\; N_2=nR_2(n)N_3,\; N_3=R_{2^8}(n+2). $$

All we need to show is that if a connected graph $G$ satisfies 
$$\#\{v\in V(G):\alpha (N(v))\ge 2\}> N_0(N_1), $$
then it contains $K_n^*, P_n, K_{1,n}^*, K_{2,n}, E_2+K_n, K_1+nP_3$ or  $CK_n$ as an induced subgraph.
By Proposition \ref {d}, $G'$ contains $P_{N_1}, \;K_{N_1}$ or $K_{1,N_1}$ as an induced subgraph.
If $G'$ contains $P_{N_1}$, we are done.
We divide it into two cases. 

Case 1: $G'$ contains $X=K_{N_1}$.

All vertices in $X$ have {\it local independence number} at least $2$ in G. For  a vertex $ x\in X$, if there doesn't exist a $ y\in G$ adjacent to $x$ but not complete to $X$, then in order to ensure $\alpha(N(x))\ge 2$, there are $y$ and $z$ complete to $K$ satisfying $y$ and $z$ are not adjacent. Hence $G$ contains $E_2+K_n$.

So we may assume that for any $x\in X$, there exists a $y$ adjacent to $x$ but not complete to $X$. Denote by $Y$ the set consisting of these corresponding $y$.
If there is a $y\in Y$ such that $|N(y)\cap X|\ge n$, then $G$ contains $E_2+K_n$. Done.
So we may assume  $|N(y)\cap X|< n, \forall y\in Y$.  Applying Lemma \ref{induced matching} to bipartite graph $G[X,Y]$, we can find $\{x_i, y_i: i\in [R_2(n)]\}$ which forms an induced matching of order $R_2(n)$ in $G[X,Y]$.  If $Y'=\{y_i: i\in [R_2(n)]\}$ contains a clique of order $n$, then $G$ contains $CK_n$. Otherwise $Y'$ contains an stable set of order $n$, then $G$ contains $K_n^*$. 

Case 2: $G'$ contains $K=K_{1,N_1}$. 

Denote the central vertex by $s$, the set of leaves by $X$. For $x\in X$, in order to ensure $\alpha(N(x))\ge 2$, one way is that  there exists a $y$ adjacent to $x$ but not $s$, and we denote by $x\in X_1$.
Otherwise another way is that there exists $y\sim z$ adjacent to $x$ and $s$, and we denote by $x\in X_2$. So $X$ is the disjoint union of $X_1$ and $X_2$.
$N_1=2N_2-1$, use pigeonhole principle.

Subcase 2.1: 
 $|X_1|\ge N_2$.

Denote by $Y$ the set consisting of vertices which is the corresponding $y$ for $x\in X_1$. If  some $y\in Y$ satisfies $|N(y)\cap X_1|\ge n$, then $G$ contains $K_{2,n}$. 
For bipartite graph $G'[X_1,Y]$, we can assume $\Delta (Y)\le n-1$.
Apply Lemma \ref{induced matching} to it, we can find $\{x_i, y_i: i\in [R_2(n)]\}$ which forms an induced matching of $G'[X_1,Y]$.
 If $Y'=\{y_i:i \in [R_2(n)]\}$ contains a clique of order $n$, then $G$ contains $K_n^*$; otherwise $Y'$ contains a stable set of order $n$, then $G$ contains $K_{1,n}^*$.

Subcase 2.2: $|X_2|\ge N_2$.

Hence we can find vertices $x_i, y_i, z_i: i\in [N_2]$ such that 
$\{s, x_i, y_i, z_i\}$ forms an induced $K_4-e$ since $y_i$ is not adjacent to $z_i$. Note that $y_i$ and $z_i$ are distinct and symmetric. However, $y_i$ maybe the same vertex as $y_j$ or $z_j$. If some vertex $s'$ truly appears $nR_2(n)$ times, then replace $s$ by $s'$. We return to subcase 2.1. For example, if $s'=z_i: i\in [nR_2(n)]$, then $s', x_i, y_i: i\in [nR_2(n)]$ is what we need in subcase 2.1.

So without loss of generality, by contracting $N_2$ to $N_3$, we can assume $y_i, z_i: i\in [N_3]$ are distinct vertices. 
Color the auxiliary complete graph $K_{N_3}$. The color of edge $(i,j),\, i<j$ is 
$$(\mathbbm 1_{x_i\sim y_j},\,
\mathbbm 1_{x_i\sim z_j},\,
\mathbbm 1_{y_i\sim x_j},\,
\mathbbm 1_{z_i\sim x_j},\,
\mathbbm 1_{y_i\sim y_j},\,
\mathbbm 1_{z_i\sim z_j},\,
\mathbbm 1_{y_i\sim z_j},\, 
\mathbbm 1_{z_i\sim y_j} ).$$

$N_3=R_{2^8}(n+2)$, so there is a monochromatic clique $K$ of order $n+2$.  Without loss of generality, we assume $V(K)=[n+2]$. Now we consider the color of $K$.

For $i,j \in [n+2],\, i<j$, suppose $\mathbbm 1_{y_i\sim y_j}=1$ first. If $\mathbbm 1_{x_i\sim y_j}=1$ then $\{x_1, x_2, y_i: i\in [3,n+2]\}$ forms an induced $E_2+K_n$. By symmetry, ($y_i$ and $z_i$ are symmetric,  $i$ and $j$ are symmetric, e.g. the arguments for $\mathbbm 1_{x_i\sim y_j}$ and $\mathbbm 1_{y_i\sim x_j}$ are symmetric.) otherwise $\mathbbm 1_{x_i\sim y_j}= \mathbbm 1_{y_i\sim x_j}=0$, $\{x_i, y_i: i\in [n]\}$ forms an induced $K_n^*$.

Next we can assume $\mathbbm 1_{y_i\sim y_j}=\mathbbm 1_{z_i\sim z_j}=0$. If $\mathbbm 1_{x_i\sim y_j}=1$, then  $\{x_1, x_2, y_i: i\in [3,n+2]\}$ forms an induced $K_{2,n}$. If $\mathbbm 1_{y_i\sim z_j}=1$, then  $\{y_1, y_2, z_i: i\in [3,n+2]\}$ forms an induced $K_{2,n}$. 
By symmetry, now we can assume 
\[\mathbbm 1_{x_i\sim y_j}=\mathbbm 1_{y_i\sim x_j}=\mathbbm 1_{x_i\sim z_j}=\mathbbm 1_{z_i\sim x_j}=\mathbbm 1_{y_i\sim z_j}=\mathbbm 1_{z_i\sim y_j}=0,\] 
hence $\{s, x_i, y_i, z_i: i\in [n]\}$ forms an induced $K_1+nP_3$.
\end{proof}

Furuya \cite{f18} considered the Ramsey-type problem for the parameter {\it domination number}.
Actually it got the following result.
\begin{thm}\label{domination number} \cite{f18}
The following are equivalent:

(1) There is a constant $c=c(\mathcal H)$ such that the {\it connected domination number} $\gamma_c(G)<c$ for every connected $\mathcal{H}$-free graph $G$.


(2) There is a constant $c=c(\mathcal H)$ such that the domination number $\gamma(G)<c$ for every connected $\mathcal{H}$-free graph $G$.

(3) $\mathcal H\le \{K_{1,n}^*, K_n^*, P_n\}$ for some positive integer $n$.

\end{thm}

\begin{proof}[Proof of Theorem \ref{local connected}]
We first prove the ``only if'' part. For graph $G$, set $G'=\{v\in V(G):c(N(v)) \ge 2\}$.
 For $G=K_n^*, P_n, K_{1,n}^*, K_{2,n}, CK_n$ or $ T_n$, the corresponding $G'$ has order $n, n-2, n+1, n+2, 2n, n+1$ respectively. Let $n=c+3$, then 
$$ \mathcal{H}\le \{K_n^*, P_n, K_{1,n}^*, K_{2,n}, CK_n, T_n \} .$$

Next we prove the ``if'' part.
By Theorem \ref{domination number}, connected $\{K_n^*, P_n, K_{1,n}^*\}$-free graph $G$ has a uniform bound $\gamma_n$ such that $\gamma(G)\le \gamma_n$.
 All we need to show is that if a connected graph $G$ satisfies 
$$\#\{v\in V(G):c(N(v))\ge 2\}>  \gamma_nR_2(N_1),\,  N_1=nN_2,\, N_2=R_2(n),$$
then it contains $K_n^*, P_n, K_{1,n}^*, K_{2,n}, T_n$ or  $CK_n$ as an induced subgraph. 

Take a minimum {\it dominating set}  $D$ of $G$, $|D|\le \gamma_n$,  $|G'|> \gamma_nR_2(N_1)$, so 
\[|G'-D|> \gamma_n(R_2(N_1)-1).\]

Hence there is a $d\in D$, such that $|N(d)\cap G'|\ge R_2(N_1)$. If $N(d)\cap G'$ contains a clique of order $N_1$, then we have a $K_{N_1}$ in $G'$. Otherwise, $N(d)\cap G'$ contains a stable set of order $N_1$, then we have a $K_{1,N_1}$ which leaves are in $G'$.

Every vertex $v$ in $G'$ satisfies $c(N(v))\ge 2$ in $G$, so we can find more vertices.

Case 1: $G'$ contains $K=K_{N_1}=\{x_i: i\in [N_1]\}$. So there is a $y_i$ adjacent to $x_i$ such that $y_i$ can not be connected to $K-x_i$ in $N(x_i)$. Hence $y_i\not =y_j$ if $i\not =j, \,\forall \,i,j\in [N_1]$. $N_1\ge R_2(n)$. If $Y=\{y_i: i\in [N_1]\}$ contains a clique of order $n$, then $G$ contains $CK_n$; otherwise $Y$ contains a stable set of order $n$, then $G$ contains $K_n^*$.

Case 2: There is a $K=K_{1,N_1}=\{s, \,x_i: i\in [N_1]\}$ such that $s$ is the central vertex, and $x_i\in G', \forall i\in [N_1]$. Hence there is a $y_i$ adjacent to $x_i$ but not $s$. If for some $i$, $|N(y_i)\cap K|\ge n$, then $G$ contains $K_{2,n}$, done. Otherwise, by Lemma \ref{induced matching}, contracting $N_1$ into $N_2$, we can assume $N(y_i)\cap K=\{x_i\}$. $N_2= R_2(n)$. If $Y=\{y_i: i\in [N_2]\}$ contains  $K_n$, then $G$ contains $T_n$. Otherwise $Y$ contains $E_n$, hence $G$ contains $K_{1,n}^*$.
\end{proof}

\begin{proof}[Proof of Theorem \ref{cut vertex}]
We first prove the ``only if'' part.  The number of cut vertices for $K_n^*, P_n, K_{1,n}^*$ are $n, n-2$ and $n+1$ respectively. Take $n=c+3$, then  $K_n^*, P_n, K_{1,n}^*$ are connected and don't satisfy the conclusion, so they are not $\mathcal{H}$-free, hence $\mathcal{H}\le \{K_n^*, K_{1,n}^*, P_n\} $.

Next we prove the ``if'' part. By Theorem \ref{domination number}, there is a constant $\gamma_c(n)$ such that every connected  $\{K_n^*, K_{1,n}^*, P_n\}$-free graph $G$ satisfies its connected domination number $\gamma_c(G)\le \gamma_c(n)$.

\lem For a connected graph $G$, the connected dominating set $D$ contains all cut vertices of $G$.

\begin{proof}Suppose not, then there is a cut vertex $v$ of $G$ such that $v\not \in D$. $G-v$ is not connected. $D\subset G-v$ is connected. So there is a connected component $C$ of $G-v$ which stays away from $D$. This contradicts to that $D$ is a dominating set. 
\end{proof}

So if $\mathcal{H}\le \{K_n^*, K_{1,n}^*, P_n\} $, then every connected $\mathcal H$-free graph $G$ is also $\{K_n^*, K_{1,n}^*, P_n\}$-free, the number of cut vertices of $G$ is bounded by $\gamma_c(n)$. 

\end{proof}
\section{Graph with many vertices of large parameter}
The following lemma is similar to the Ramsey's theorem. As for proof, we may assume $k=2$, and consider the bipartite case. Then use the same technique as the classical proof of Ramsey's theorem.

\begin{lem}\cite{a} \label{bipartite Ramesy }
 For any positive integer $k$ and $q$, there exists a number $MR(k,q)$ such that in every $k$-partite graph $G=(V_{1},V_{2},\dots , V_{k}, E) $ with $|V_i|\ge MR(k,q), \forall i$, there is a collection of subsets $U_i\subset V_i$ of order $|U_i|=q$ satisfying every pair of subsets  induces either a {\it biclique} $K_{q,q}$ or its bipartite complement $E_{2q}$.
 \end{lem}

\begin{proof}[Proof of Theorem \ref{finite, deg}]
We first prove the ``only if'' part.
 Set $n=c_1+c_2$, then for $G=K_n, K_{n,n}, nK_{1,n}$, the number of vertices with degree at least $c_1$ is $n, 2n, n$ respectively. So they are not $\mathcal{H}$-free, hence $\mathcal{H}\le \{K_n, K_{n,n}, nK_{1,n}\}$.

Next we prove the ``if'' part. 

Claim: Set $N_3=R_2(n),\;  N_2=R_{2^{n^2+2n+1}}(2n),\;N_1=N_2\cdot N_3+N_2$, then for every $\{K_n, K_{n,n}, nK_{1,n}\}$-free graph $G$, we have 
\[\#\{v\in V(G):\deg(v)\ge N_1\} <N_2. \]

Suppose not, name the vertices with large degree by $v_i, i\in [N_2]$. $\deg(v_i)\ge N_2\cdot N_3+N_2$,  so for $v_i$, we can find $v_i^j: j\in [N_3]$ adjacent to $v_i$. All $v_i^j: i\in [N_2],\, j\in [N_3]$  are distinct even though $v_i$ is adjacent to $v_{i'}$ and $v_{i'}^j$. For every $i$, we can find $n$ stable vertices from $\{v_i^j: j\in [N_3]\}$ since there is no $K_n$. For convenience, we denote these $n$ stable vertices by $\{v_i^j: j\in [n]\}$.  
Color the auxiliary complete graph $K_{N_2}$. The edge $(i,i'),\, i<i'$ is colored by 

\[ (\mathbbm 1_{v_i\sim v_{i'}}, \mathbbm 1_{v_i\sim v_{i'}^j}, \mathbbm 1_{v_{i'}\sim v_{i}^j}, \mathbbm 1_{v_i^j\sim v_{i'}^{j'}} :j,j'\in [n])\]

We use $2^{n^2+2n+1}$ colors, $N_2=R_{2^{n^2+2n+1}}(2n)$, so we can find a monochromatic clique $K$ of order $2n$. For convenience, suppose $V(K)=[2n]$. Consider the color of $K$.

$G$ is $K_n$-free, so $\mathbbm 1_{v_i\sim v_{i'}}= \mathbbm1_{v_i^j\sim v_{i'}^{j}}=0$. 
If $\mathbbm 1_{v_i^j\sim v_{i'}^{j'}}=1$ for some distinct $j,j' \in[n]$ then $\{v_1^j,\dots ,v_n^j; v_{n+1}^{j'},\dots ,v_{2n}^{j'}\}$ forms a $K_{n,n}$, contradiction.   So $\mathbbm 1_{v_i^j\sim v_{i'}^{j'}}=0$. If $\mathbbm 1_{v_i\sim v_{i'}^j}=1$ for some $j$, then $\{v_1,\dots ,v_n; v_{n+1}^{j},\dots ,v_{2n}^j\}$ forms a $K_{n,n}$, a contradiction. 
By symmetry, all the characteristic function is $0$. Thus we have a $nK_{1,n}$.

Therefore the claim is true. Set $c_1=N_2\cdot N_3+N_2, \; c_2=N_2$, we complete the proof.
\end{proof}

\begin{proof}[Proof of Theorem \ref{finite, local stable number}]
It's clear that ``$(1)\Rightarrow (2)$'' holds sine $\alpha (N(v))\ge c(N(v))$.

We show that ``$(2)\Rightarrow (3)$'' holds.
Set $n=c_1+c_2$, then $\#\{v\in V(G): c(N(v))\ge c_1\}$ for $G=K_{n,n}, nK_{1,n}, K_n+E_n, K_n^n$ is $2n, n, n, n$ respectively. All are larger than $c_2$. So they are not $\mathcal{H}$-free. 
Hence  $\mathcal{H}\le \{K_{n,n}, nK_{1,n}, K_n+E_n, K^n_n\}$.

Finally We show that ``$(3)\Rightarrow (1)$'' holds.
Set $$N_2=R_{2^{2n+1}}(2n),\; N_3=MR(N_2, 3n),\;N_1=N_2\cdot N_3+N_2.$$
We only need to show that if a graph $G$ satisfies 
\[ \#\{v\in V(G): \alpha(N(v))\ge c_1=N_1\} \ge c_2=N_2 \]

then it contains  $K_{n,n}, nK_{1,n}, K_n+E_n$ or $ K^n_n$.

Take $N_2$ vertices $v_i, i\in[N_2]$ satisfying $\alpha(N(v_i))\ge c_1$. So we can find $N_3$ stable neighbors of $v_i$ separately. By Lemma \ref{bipartite Ramesy }, since $N_3=MR(N_2,3n)$, we can require all these neighbors of $v_i, i\in [N_2]$ are stable, i.e. we have stable set $\{v_i^j: i\in [N_2], j\in [3n]\}$ satisfying $v_i\sim v_i^j$.  Color the  auxiliary complete graph $K_{N_2}$. The edge $(i,i'),\, i<i'$ is colored by 
\[(\mathbbm 1_{v_i\sim v_{i'}}, \mathbbm 1_{v_i\sim v_{i'}^j }, \mathbbm 1_{v_{i'}\sim v_{i}^j}: j\in [3n] ).\]

Then we can get a monochromatic clique $K$ of order $2n$. Without loss of generality, we may assume $V(K)=[2n]$.
We now handle the color of $K$. Suppose firstly $\mathbbm 1_{v_i\sim v_{i'}}=0$. If $\mathbbm 1_{v_i\sim v_{i'}^j }, \mathbbm 1_{v_{i'}\sim v_{i}^j}: j\in [3n] $ are all $0$, then $G$ contains $nK_{1,n}$. Otherwise, by symmetry, there exists some $j$ such that $\mathbbm 1_{v_i\sim v_{i'}^j }=1$, then $\{v_i:i\in [n], v_{i'}^j : i'\in [n+1,2n]\}$ forms a $K_{n,n}$.

Suppose now $\mathbbm1_{v_i\sim v_{i'}}=1$. If $J=\{j\in [3n]:\mathbbm 1_{v_i\sim v_{i'}^j }=1\}$ has more than $n-1$ elements, ($|J|\ge n$.) then $\{v_i: i\in [n], v_{n+1}^j: j\in J\}$ contains a $K_n+E_n$. By symmetry, we assume $J$ and $J'=\{j\in [3n]:1_{v_{i'}\sim v_{i}^j }=1\}$ have no more than $n-1$ elements. So every $v_i$ have $n$ leaves not adjacent to other $v_{k},\,\forall k\not =i$. 
Hence $G$ contains a $K_n^n$ as an induced subgraph.
\end{proof}

\begin{proof}[Proof of Theorem \ref{finite adh}]
We first prove the ``only if'' part.
 Set $n=c_1+c_2$, then for $G=nK_{1,n}, K^n_n$, the number of vertices with adhesion at least $c_1$ is $n, n$ respectively. So they are not $\mathcal{H}$-free, hence $\mathcal{H}\le \{ nK_{1,n}, K^n_n\}$.

Next we prove the ``if'' part. 

Claim: Set $N_2=R_{2}(n),\; N_3=n+N_2-1,\;N_1=N_2\cdot N_3+N_2$, then for every $\{nK_{1,n}, K_n^n\}$-free graph $G$, we have 
\[\#\{v\in V(G): \operatorname{adh}(v)\ge N_1\} <N_2. \]

Suppose not, we can find $N_2$ vertices $\{v_i: i\in [N_2]\}$ of large adhesion. Hence there is $\{v_i^j: j\in [N_3]\} \subset N(v_i)$ for any $i$ such that $v_i^j$ can be connected to $v_i^{j'}$ only through $v_i$. If there are two vertices $v_i^j$ and $v_i^{j'}$ satisfying $N(v_i^j)\cap \{v_{i'}, v_{i'}^k: k\in [N_3]\}\not =\emptyset $ and $N(v_{i}^{j'})\cap \{v_{i'}, v_{i'}^k: k\in [N_3]\}\not =\emptyset $, then it is contradicted to that $v_i^j$ can be connected to $v_i^{j'}$ only through $v_i$. So by change $N_3$ into $N_3-(N_2-1)=n$, we may assume 
\[N(v_i^j)\cap \{v_{i'}, v_{i'}^k: k\in [n]\} =\emptyset , \forall i\not = i', \forall j\in [n].\] 
The set $V=\{v_i: i\in [N_2]\}$ has order $N_2=R_2(n)$. So if $V$ has a clique of order $n$, then $G$ contains $K_n^n$, otherwise $V$ has a stable set of order $n$, then $G$ contains $nK_{1,n}$. The claim is complete. Set $c_1=N_1,\, c_2=N_2$, the proof is complete. 
\end{proof}

\section{Concluding remarks}
In Theorem \ref{cut vertex}, we require that the $\mathcal{H}$-free graph $G$ is connected. If we drop this condition, we can find stronger $\mathcal{H}$, so the problem will be easier. We have the following corollary.

\begin{cor} There is a constant $c=c(\mathcal H)$ such that 
$\#\{v\in V(G):\deg(v)\ge 2\}\le c $ 
for every  $\mathcal {H}$-free graph $G$ if and only if $\mathcal{H}\le \{K_n, nP_3, nK_3, K_{1,n}^*, K_{2,n}, K_2+nK_1, K_1+nK_2 \} $ for some $n$.
\end{cor}

\begin{proof} Compare with the corresponding connected version.
We first prove the ``only if'' part.

Suppose $\mathcal H$ satisfies there is a constant $c=c(\mathcal H)$ such that $\#\{v\in V(G):\deg(v)\ge 2\}\le c $ for every  $\mathcal {H}$-free graph $G$. The number of vertices with degree at least $2$ for 
$$K_n, nP_3, nK_3, K_{1,n}^*, K_{2,n}, K_2+nK_1, K_1+nK_2$$
are $n, n, 3n, n+1, n+2, n+2, 2n+1 $ respectively. Let $n= c+1$, then all these graphs satisfy the number of vertices with degree at least $2$ is not bounded by $c$, so they are not $\mathcal{H}$-free. Hence $\mathcal{H}\le \{K_n, nP_3, nK_3, K_{1,n}^*, K_{2,n}, K_2+nK_1, K_1+nK_2 \} $.

Next we prove the ``if'' part. 
Suppose 
$$\mathcal{H}\le \{K_n, nP_3, nK_3, K_{1,n}^*, K_{2,n}, K_2+nK_1, K_1+nK_2 \}. $$ 
Then every $\mathcal H$-free graph $G$ is also $\{K_n, nP_3, nK_3, K_{1,n}^*, K_{2,n}, K_2+nK_1, K_1+nK_2 \} $-free. The connected component of $G$ which contains a vertex of degree larger than  $2$ must contain an induced $P_3$ or $K_3$. Since $G$ is $\{nP_3, nK_3\}$-free, there are at most $2n-2$ connected components of $G$ which contains a vertex of degree larger than $1$.  
For a connected component $C$, since $C$ is $\{K_{4n}, P_{4n}, K_{1,4n}^*, K_{2,4n}, K_2+4nK_1, K_1+4nK_2 \} $-free, by Theorem \ref{deg}, $$\#\{v\in V(C):\deg(v)\ge 2\}\le N_0(N_1), \, N_1=(4n-1)(R_2(8n-1)),$$ 
hence $\#\{v\in V(G):\deg(v)\ge 2\}\le (2n-2)N_0(N_1)$. Done.
\end{proof}

Use the similar technique, we can prove the following. 
\begin{cor}
There is a constant $c=c(\mathcal H)$ such that 
$$\#\{v\in V(G):\alpha(N(v)) \ge 2\} \le c $$

 for every $ \mathcal {H}$-free graph $G$ if and only if 
 $$ \mathcal{H}\le \{K_n^*, nP_3, K_{1,n}^*, K_{2,n}, E_2+K_n, CK_n \} $$ for some $n$.
\end{cor}

\begin{cor}There is a constant $c=c(\mathcal H)$ such that 
$$\#\{v\in V(G):c(N(v)) \ge 2\} \le c $$
 for every  $\mathcal {H}$-free graph $G$ if and only if 
 $$ \mathcal{H}\le \{K_n^*, nP_3, K_{1,n}^*, K_{2,n}, CK_n, T_n \} $$for some $n$.
\end{cor}

\begin{cor}There is a constant $c=c(\mathcal H)$ such that 
$$\#\{v\in V(G): \operatorname{adh}(v) \ge 2\} \le c $$
 for every  $\mathcal {H}$-free graph $G$ if and only if 
 $ \mathcal{H}\le \{K_n^*, nP_3, K_{1,n}^*\} $ for some $n$.
\end{cor}
Note that the proof previously appeared can be adjusted slightly and the bound can be made smaller. For example, we can use Theorem \ref{domination number} to handle theorems in section $2$ as the same way as Theorem \ref{local connected}.
We can  require that the graphs in $\mathcal H$ use different positive integer, such as $\mathcal H\le \{K_{1,n_1}, K_{n_2}, P_{n_3}\}$ in Theorem \ref{cut vertex}. We can also require that the constant $c_1(\mathcal H)$ and $c_2(\mathcal H)$ be the same in Theorem \ref{finite, deg} since we can replace  $c_1(\mathcal H)$ and $c_2(\mathcal H)$ by the same constant $c_1(\mathcal H)+c_2(\mathcal H)$.

 However, it has no meaning except making proofs seem more complicated.
As for Theorem \ref{finite, deg}, Theorem \ref{finite, local stable number} and Theorem \ref{finite adh}, we don't require $G$ is connected, because there is no typical way to connect the $n$ discrete $K_{1,n}$. For a tree with many leaves, replace the leaves by $K_{1,n}$, then we get a counterexample. However, we have no good way to forbid this case. The more explanation, see \cite{c20}, \cite{l17}.\bigskip


\end{document}